\documentclass[12pt]{article}
\usepackage{amsmath,amssymb,amsthm,color}
\usepackage{scrtime}

\renewcommand\le{\leqslant}
\renewcommand\ge{\geqslant}
\newcommand\none[1]{\|#1\|_1}

\definecolor{blue}{rgb}{0,0,1}
\definecolor{green}{rgb}{0,1,0}
\definecolor{red}{rgb}{1,0,0}
\newcommand\g{\color{green}}
\newcommand\rd{\color{red}}
\newcommand\bk{\color{black}}

\newcommand\pc{p_{\mathrm{c}}}

\newcommand\Z{\mathbb{Z}}
\newcommand\Prb{\mathbb{P}}
\newcommand\E{\mathbb{E}}
\newcommand\cL{\mathcal{L}}
\newcommand\cR{\mathcal{R}}
\newcommand\cP{\mathcal{P}}
\newcommand\cE{\mathcal{E}}
\newcommand\Fg[1]{Figure~\ref{f:#1}}

\newcommand\Binf{B^{\infty}}
\newcommand\Bone{B^{1}}
\newcommand\Sinf{S^{\infty}}
\newcommand\Sone{S^{1}}
\newcommand\floor[1]{\lfloor#1\rfloor}

\theoremstyle{plain}
\newtheorem{theorem}{Theorem}
\newtheorem{lemma}[theorem]{Lemma}
\newtheorem{conjecture}[theorem]{Conjecture}
\theoremstyle{definition}
\newtheorem{claim}[theorem]{Claim}

\newdimen\unit\newdimen\psep\newcount\nd\newcount\ndx\newbox\dotb
\newdimen\dx\newdimen\dy\newdimen\dxx\newdimen\dyy\newdimen\hgt
\newcommand\clap[1]{\hbox to 0pt{\hss{#1}\hss}}
\newcommand\vdisk[1]{{\setbox0\clap{\font\dotf=cmr10 scaled #1\dotf.}%
\raise-.5\ht0\box0}}
\newcommand\vblob[1]{{\setbox0\clap{$#1$}\raise-.55\ht0\box0}}
\newcommand\varline[2]{\setbox\dotb\hbox{\vdisk{#1}}\psep=#2\ht\dotb}
\newcommand\point[3]{\rlap{\kern#1\unit\raise#2\unit\hbox{#3}}}
\newcommand\setnd[4]{\dx=#3\unit\advance\dx-#1\unit\divide\dx by\psep
\dy=#4\unit\advance\dy-#2\unit\divide\dy by\psep
\multiply\dx by\dx\multiply\dy by\dy\advance\dx\dy\nd=1
\loop\ifnum\dx>0\advance\dx-\nd sp\advance\nd1\advance\dx-\nd sp\repeat}
\newcommand\dl[4]{{\setnd{#1}{#2}{#3}{#4}\dline{#1}{#2}{#3}{#4}\nd}}
\newcommand\dline[5]{{\nd=#5\hgt=#2\unit\dx=#3\unit\advance\dx-#1\unit
\divide\dx by\nd\dy=#4\unit\advance\dy-#2\unit\divide\dy by\nd
\rlap{\kern#1\unit\loop\ifnum\nd>1\advance\nd-1\advance\hgt\dy
\kern\dx\raise\hgt\copy\dotb\repeat}}}
\newcommand\ptlr[3]{\point{#1}{#2}{\raise-.4ex\rlap{$\ \,\scriptstyle{#3}$}}}
\newcommand\ptll[3]{\point{#1}{#2}{\raise-.4ex\llap{$\scriptstyle{#3}\ $}}}
\newcommand\pt[2]{\point{#1}{#2}{\vblob{\bullet}}}
\newcommand\pe[2]{\point{#1}{#2}{\vblob{\circ}}}
\newcommand\px[2]{\point{#1}{#2}{\tiny\vblob{\blacksquare}}}
\newcommand\thnline{\varline{600}{.4}}
\newcommand\medline{\varline{1200}{.3}}
\newcommand\dotline{\varline{600}{8}}

\title{Essential enhancements revisited}

\author{Paul Balister%
\thanks{Department of Mathematical Sciences, University of Memphis, Memphis TN 38152, USA.
E-mail: {pbalistr@memphis.edu}.}
\thanks{Research supported in part by NSF grant DMS-1301614.}
\and B\'ela Bollob\'as%
\thanks{Department of Pure Mathematics and Mathematical Statistics,
Wilberforce Road, Cambridge CB3 0WB, UK and
Department of Mathematical Sciences, University of Memphis, Memphis TN 38152, USA.
E-mail: {\tt b.bollobas@dpmms.cam.ac.uk}.}
\thanks{Research supported in part by EU MULTIPLEX grant 317532.}${\ }^\dag$
\and Oliver Riordan%
\thanks{Mathematical Institute, University of Oxford, Radcliffe Observatory Quarter, Woodstock Road, Oxford OX2 6GG, UK.
E-mail: {\tt riordan@maths.ox.ac.uk}.}}
\date{February 4, 2014}
\begin{document}

\maketitle

\begin{abstract}
In 1991 Aizenman and Grimmett claimed that any `essential enhancement' of
site or bond percolation on a lattice lowers the critical probability,
an important result with many implications, such as strict
inequalities between critical probabilities on suitable pairs of lattices.
Their proof has two parts, one probabilistic and one combinatorial. In this
paper we point out that a key combinatorial lemma, for which they provide only
a figure as proof, is false. We prove an alternative form of the lemma, and thus
the enhancement result, in the special cases of site percolation on the square, triangular
and cubic lattices, and for bond percolation on $\Z^d$, $d\ge 2$. The general case remains open,
even for site percolation on $\Z^d$, $d\ge 4$.
\end{abstract}

\section{Introduction}

In \emph{(independent) site percolation} on $\Z^d$, $d\ge 2$,
each site (vertex) of the lattice $\Z^d$ is taken to be \emph{open}
with probability $p$ and \emph{closed} otherwise, independently
of the others. Let $\omega$ be the resulting \emph{configuration},
or random set of open sites, and $\Omega=\cP(\Z^d)$ the set of possible configurations.
The basic question in percolation is when $\omega$ (or, more precisely, the subgraph
of $\Z^d$ induced by $\omega$) contains an infinite cluster (component).
By Kolmogorov's $0$/$1$-law, the probability of this event is $0$ or $1$
for any given $p$ and, following Broadbent and Hammersley~\cite{BH}, one defines
the \emph{critical probability} $\pc$ for site percolation on $\Z^d$
to be the infimum of the set of $p$ for which this probability is $1$.
The definition extends naturally to other lattices, and to \emph{bond percolation},
where it is the edges rather than the vertices that are open or closed.
For further background and definitions see~\cite{BRbook,Grimmett}.

Informally speaking, an enhancement of site percolation on $\Z^d$
is a local rule that `adds in' extra open sites depending on the existing
configuration, in a translation-invariant way.
Let $\delta$ be a distance on $\Z^d$, for example the
$\ell_\infty$ or the $\ell_1$ metric, and let $B_r=B_r(0)=\{v:\delta(0,v)\le r\}$
be the corresponding closed ball of radius $r$.
(When we wish to specify the $\ell_\infty$ or $\ell_1$ metric we write $B_r^\infty$
or $B_r^1$.)
An \emph{enhancement} is defined by a function
$\cE_0:\Omega\to\Omega$ that is \emph{local} in that

(i) there is some
$r$ such that $\cE_0(\omega)$ depends only on $\omega\cap B_r$, and

(ii) $\cE_0(\omega)$ is always finite.

\noindent
Increasing $r$ if necessary
we may assume in addition that

 (ii') $\cE_0(\omega)\subseteq B_r$.

\noindent
We say that the enhancement $\cE_0$ has \emph{range} $r$
if (i) and (ii') hold.

Let $\alpha$ be a second random configuration, obtained by including
each site with probability $s$ independently of the others and of $\omega$,
and define the \emph{enhanced configuration} to be
\begin{equation}\label{e:E}
 \cE(\omega,\alpha)
 :=\omega\cup\bigcup_{v\in\alpha} ( \cE_0(\omega-v)+v),
\end{equation}
where, as usual, $\omega\pm v$ denotes the translate of the configuration
$\omega$ through the vector $\pm v$.
The significance of $s$ (often taken to be $1$) and of $\alpha$ is that the enhancement
rule is only \emph{activated} at a given site with probability $s$.

A simple example of an enhancement would be to add into $\omega$
any site $v$ such that $v$ and its neighbours are all closed (i.e., not included
in $\omega$). Another would be to add into $\omega$ all the neighbours of any site
that is open but has no open neighbours.

A natural question is when an enhancement (with $s=1$, or $s>0$ fixed)
lowers the critical probability, although one must be a little careful with the definition,
since the rule $\cE_0(\omega)$ need not be monotone in $\omega$,
and so the (law of) $\cE(\omega,\alpha)$
need not be monotone in $p$.

Following Aizenman and Grimmett~\cite{AG} we call an enhancement \emph{essential}
if there is a configuration $\omega$ such that $\omega$ does not contain
a doubly infinite path, but $\cE(\omega,\{0\})=\omega\cup \cE_0(\omega)$
does. In other words, an enhancement is essential if it is \emph{possible}
that activating the rule at a single site can create a doubly infinite path.
For example, the second enhancement described above is essential, while the first
is not.
The definitions given so far extend to bond percolation and to other lattices in a natural way;
we omit the details since we shall focus mostly on site percolation on $\Z^d$.

Let $\theta(p,s)$ denote the probability that in the enhanced configuration $\cE(\omega,\alpha)$
the origin is in an infinite cluster, noting that
$\theta(p,s)>0$ if and only if $\cE(\omega,\alpha)$ contains an infinite cluster with probability $1$.
The main result of Aizenman and Grimmett~\cite{AG} may be stated as follows;
we discuss the reason for calling it a conjecture below.

\begin{conjecture}\label{cmain}
Let $\cE_0$ be an essential enhancement of site or bond percolation on
a lattice $\cL$ with critical probability $\pc>0$.
Then for any $s>0$ there is a $\pi(s)<\pc$ such that $\theta(p,s)>0$ for
all $p$ satisfying $\pi(s)<p<\pc$.
\end{conjecture}

Note that it is not entirely clear in~\cite{AG} what the scope of the claimed result
is; a formal result is stated only for site percolation on $\Z^d$, but later the authors say
that this restriction is not essential. They also mention that the choice of lattice is `to a large extent' irrelevant to their arguments.

\medskip
In this paper we have two aims. First, we shall describe in what way the proof of Conjecture~\ref{cmain}
given by Aizenman and Grimmett~\cite{AG} is incomplete. Second, we shall prove the following special case.

\begin{theorem}\label{thmain}
Conjecture~\ref{cmain} holds for site percolation on $\Z^2$ and on $\Z^3$.
\end{theorem}

We shall also outline a proof of the following much easier result, without giving full details.

\begin{theorem}\label{thextra}
Conjecture~\ref{cmain} holds (\emph{mutatis mutandis}) for bond percolation on $\Z^d$, $d\ge 2$,
and for site percolation on the (two-dimensional) triangular lattice.
\end{theorem}

We believe Conjecture~\ref{cmain} very strongly for site percolation on $\Z^d$, $d\ge 4$,
and slightly less strongly for arbitrary lattices; it does not
seem inconceivable that there is a pathological counterexample.

The rest of the paper is organized as follows. In the next section
we outline the argument given by Aizenman and Grimmett~\cite{AG},
and describe the problem with it. In Section~\ref{sec_closing}
we formulate a condition (Conjecture~\ref{cpath}) that
would imply Conjecture~\ref{cmain}. In Section~\ref{sec_inside}
we give (for completeness) the proof of a non-problematic
lemma from~\cite{AG}. In Section~\ref{sec_Z2Z3} we
prove Theorem~\ref{thmain} by proving the corresponding
cases of Conjecture~\ref{cpath}. Finally, we briefly discuss
Theorem~\ref{thextra} in Section~\ref{sec_further}.

\section{The Aizenman--Grimmett argument}\label{sec_AG}

The argument for Conjecture~\ref{cmain} given in~\cite{AG} has
two parts: a probabilistic part that is correct and (reasonably)
complete, and a combinatorial part that is neither. We describe
both parts, starting with an outline of the probabilistic part.

For $L\ge 1$ let $S_L=S_L(0)=\{v:\delta(0,v)=L\}$ be the \emph{sphere} of radius $r$
centred at the origin;
if we wish to specify the metric we write $S^1_L$ or $S^\infty_L$.
Let $\cR_L$ be the event (subset of $\Omega$, measurable
with respect to the standard product $\sigma$-algebra) that there is an (open) path connecting the
origin to $S_L$. Finally, let
$\tau_L(p,s)$ be the probability that $\cE(\omega,\alpha)\in \cR_L$,
i.e., that the enhanced configuration contains a path from $0$ to $S_L$.
Note that
\[
 \theta(p,s)=\lim_{L\to\infty} \tau_L(p,s).
\]

Let us call a site $v$ {\em $p$-pivotal\/} (for the event $\cR_L$, and
with respect to the state $(\omega,\alpha)$)
if $\cE(\omega\cup\{v\},\alpha)$ contains a path from 0 to $S_L$ but
$\cE(\omega\setminus\{v\},\alpha)$ does not.
We say that $v$ is {\em $s$-pivotal\/} if
$\cE(\omega,\alpha\cup\{v\})$ contains a path from 0 to $S_L$ but
$\cE(\omega,\alpha\setminus\{v\})$ does not. In~\cite{AG}, the
terms \emph{$(n+)$pivotal} and \emph{$(a+)$pivotal} are used.
Note that a site can be pivotal in another sense (\emph{$(n-)$pivotal}
in~\cite{AG}) in that deleting it from $\omega$ causes $\cR_L$
to hold; we will not need to consider this case.
Let
\[
 N_p(L) = N_p(L,\omega,\alpha) = \big|\{v: v\hbox{ is $p$-pivotal }\}\big|
\]
and
\[
 N_s(L) = N_s(L,\omega,\alpha) = \big|\{v: v\hbox{ is $s$-pivotal }\}\big|
\]
be the numbers of $p$-pivotal and $s$-pivotal sites, respectively.

The argument in \cite{AG} is based on the following idea.
Writing $\Prb_{p,s}$ for the probability measure on $\Omega\times\Omega$
defined above, and $\E_{p,s}$ for the corresponding expectation,
suppose there exist an integer $L_0$ and a continuous, strictly
positive function $g(p,s)$ on $(0,1)^2$ such that for
any $L\ge L_0$ and $(p,s)\in (0,1)^2$ we have
\begin{equation}\label{EE}
 \E_{p,s} N_s(L) \ge g(p,s) \E_{p,s} N_p(L).
\end{equation}
Then, using a version of the Margulis--Russo formula, it is easy
to see that
\begin{equation}\label{pdineq}
 \frac{\partial}{\partial s} \tau_L(p,s) \ge g(p,s)\frac{\partial}{\partial p} \tau_L(p,s).
\end{equation}
Aizenman and Grimmett take the limit to transfer this inequality
to one for $\theta(p,s)$. More precisely, \eqref{pdineq} shows
that $\tau_L(p,s)$ is weakly decreasing as $s$ decreases and $p$ increases
along certain specific curves that do not depend on $L$.
This implies that $\theta(p,s)$ is monotone along the same curves; then the fact that
$\theta(p,0)>0$ for $p>\pc$ gives the result.

To establish \eqref{EE} it suffices to prove the following statement.
\begin{claim}\label{cl1}
Let $\cE_0$ be an essential enhancement.
There are constants $L_0$ and $R$, depending only on $\cE_0$,
such that whenever $L\ge L_0$ and a site $v$ is $p$-pivotal for $\cR_L$ in the
state $(\omega,\alpha)$, then there is a state $(\omega',\alpha')$
differing from $(\omega,\alpha)$ only within distance $R$ of $v$
such that some site $w$ within distance $R$ of $v$ is $s$-pivotal for $\cR_L$
in the state $(\omega',\alpha')$.
\end{claim}
In other words, if $v$ is $p$-pivotal, then it is possible to modify
the configurations $\omega$ and $\alpha$ within a fixed distance of $v$
to make some site near $v$ (usually this can be $v$ itself) $s$-pivotal.
Indeed, such a modification procedure gives a finite-to-one map from $\Omega\times\Omega$
to itself; considering this map, and taking expectations, it is easy to deduce \eqref{EE}.

Note that Claim~\ref{cl1} is a purely combinatorial (graph theoretic) statement
about subgraphs of $\Z^d$ and whether they do or do not contain paths with
certain properties; the probability measure $\Prb_{p,s}$ does not appear.

Aizenman and Grimmett observe that in proving Claim~\ref{cl1}, we may
assume that enhancements near $v$ are already deactivated, i.e.,
it suffices to prove the following modified claim.
\begin{claim}\label{cl2}
Let $\cE_0$ be an essential enhancement. 
There are constants $L_0$ and $R$, depending only on $\cE_0$,
such that whenever $L\ge L_0$ and a site $v$ is $p$-pivotal for $\cR_L$ in the
state $(\omega,\alpha)$, and $\alpha\cap B_R(v)=\emptyset$,
then there is a state $(\omega',\alpha')$
differing from $(\omega,\alpha)$ only within distance $R$ of $v$
such that some site $w$ within distance $R$ of $v$ is $s$-pivotal for $\cR_L$
in the state $(\omega',\alpha')$.
\end{claim}
Indeed, suppose we are given a state $(\omega,\alpha)$ with $v$ $p$-pivotal.
Then $\cE(\omega\cup\{v\},\alpha)$ contains a path from $0$ to $S_L$
and $\cE(\omega\setminus\{v\},\alpha)$ does not.
Consider deleting points $w\in \alpha\cap B_R(v)$ from $\alpha$ one-by-one,
which can only remove sites from the enhanced configuration.
If at some stage there is no longer a path from $0$ to $S_L$
in the current configuration $\cE(\omega\cup\{v\},\alpha')$,
then the last activation site $w$ deleted is $s$-pivotal in this state,
and we have found a configuration
of the type required by Claim~\ref{cl1}. Otherwise,
$v$ is $p$-pivotal in $(\omega,\alpha\setminus B_R(v))$ and we apply
Claim~\ref{cl2}.

The strategy of the proof of Claim~\ref{cl2} described in~\cite{AG} is very simple.
We state the first part as a lemma. Here, following~\cite{AG},
$e_m$ denotes the site $(m,0,\ldots,0)$; we write $-e_m$ for $(-m,0,\ldots,0)$.
There are two version of this result, one for the $\ell_1$ metric and one for $\ell_\infty$.
Either immediately implies the other.

\begin{lemma}\label{inside}
Let $\cE_0$ be an essential enhancement of site percolation on $\Z^d$, $d\ge 2$, with range $r$. Then there is an $m>r$
and a (finite) configuration $\omega$ with the following properties:

(i) $\omega\subset B_m$, and $\omega\cap S_m=\{e_m,-e_m\}$.

(ii) $\omega$ does not contain a path joining $-e_m$ to $e_m$.

(iii) $\omega\cup \cE_0(\omega)$ does contain a path joining $-e_m$ to $e_m$.
\end{lemma}
In other words, in the state $(\omega,\emptyset)$, the origin is $s$-pivotal for the event
that there is a path
joining $\pm e_m$. This configuration serves as a stand-alone unit that can be `plugged in'
to some more complicated configuration to create an $s$-pivotal site:
since $\cE_0(\omega)\subseteq B_{m-1}$, activating the enhancement at the origin has no effect
on or outside $S_m$, so from the outside the only change is that the points $\pm e_m$
are now connected inside $B_m$, which previously they were not.

Although Lemma~\ref{inside} is essentially immediate, we provide a complete proof
in Section~\ref{sec_inside}; Aizenman and Grimmett do not give a proof.

The second part of the strategy in~\cite{AG} is as follows. Suppose that
in the state $(\omega,\alpha)$ the site $v$ is
$p$-pivotal for the event $\cR_L$ that there is an open path from $0$ to $S_L$,
with no points of $\alpha$ near $v$.
Suppose for the moment that $v$ is far from both $0$ and $S_L$.
Then there is an open path $P=v_0v_1\cdots v_\ell$ joining $0$ to $S_L$ and passing through $v$.
The idea is to pick a suitable $r>m$ (they take $r=m+4$ where $m$ is as in Lemma~\ref{inside}),
and to modify $P$ within $B^\infty_r(v)$
so that inside $B^\infty_m(v)$ (or $B^1_m(v)$ -- here it makes little difference) the configuration
is a translate of that given by Lemma~\ref{inside}, and $v$ is $s$-pivotal.

Let $v_i$ and $v_j$ be the first and last points of $P$ in $B^\infty_r(v)$.
Aizenman and Grimmett~\cite[p.\ 829]{AG} state that it is possible to change the configuration
$\omega$ inside $B^\infty_r(v)$ so that there are open paths from $v_i$
to $v-e_m$ and from $v_j$ to $v+e_m$ contained `strictly within'
$B^\infty_r(v)\setminus B^\infty_m(v)$ (which we interpret to mean
within $B^\infty_{r-1}\setminus B^\infty_m(v)$)
except for their endvertices, such that
no vertex of the first is a neighbour of a vertex of the second. This statement
would, together with Lemma~\ref{inside}, imply Claim~\ref{cl2} and hence Conjecture~\ref{cmain}.
Rather than give a `turgid formal proof' of the existence of these paths,
they refer the reader to a figure similar to \Fg{1}.
Unfortunately, such paths do not in general exist.

\begin{figure}
\[\unit=12pt
 \thnline
 \dl{-8}{-8}{8}{-8}\dl{8}{-8}{8}{8}\dl{8}{8}{-8}{8}\dl{-8}{8}{-8}{-8}
 \dl{-2}{-2}{2}{-2}\dl{2}{-2}{2}{2}\dl{2}{2}{-2}{2}\dl{-2}{2}{-2}{-2}
 \dotline
 \dl{-5}{-5}{5}{-5}\dl{5}{-5}{5}{5}\dl{5}{5}{-5}{5}\dl{-5}{5}{-5}{-5}
 \medline
 \dl{-9}{7}{-6}{7}\dl{-6}{7}{-6}{0}\dl{-6}{0}{-3}{0}\dl{-3}{0}{-3}{1}\dl{-3}{1}{-2}{1}
 \dl{-9}{6}{-7}{6}\dl{-7}{6}{-7}{-7}\dl{-7}{-7}{7}{-7}\dl{7}{-7}{7}{0}
 \dl{7}{0}{3}{0}\dl{3}{0}{3}{3}\dl{3}{3}{1}{3}\dl{1}{3}{1}{2}
 \pt00\ptlr{0}{0}{v}\ptlr{-2}{-4}{B_m}\ptll{8}{7}{B_r}
\]
\caption{Rewiring paths in an annulus. The configuration inside $B_m$ is given
by Lemma~\ref{inside}.}\label{f:1}
\end{figure}

For example, suppose $P$
first enters the cube $\Binf_r(v)$ at a corner point. There is no
way that this path can be continued in the {\em interior} of $\Binf_r(v)$.
Allowing ourselves to continue in the boundary $\Sinf_r(v)$ of $\Binf_r(v)$ does not help.
Indeed, thinking of the vertices $v_0,\ldots,v_i$  as red and $v_j,\ldots,v_\ell$
as green, since we aim to make the site $v$ $s$-pivotal, we must avoid
creating a red-green connection that does not go through $\Binf_m(v)$.
Continuing the red path along the boundary of $\Binf_r(v)$ a few steps before entering
the interior may thus be ruled out by the presence of certain green vertices
in $\Sinf_{r+1}(v)$; see \Fg{2}. A somewhat tedious
case-by-case analysis is possible to resurrect the proof in two dimensions,
by redirecting the green path to enter $\Binf_r(v)$ at an earlier point, but
in three or more dimensions the situation becomes much worse. Indeed,
there are problems with paths meeting $\Binf_r(v)$ not just at the
corners, but at points along edges as well.

\begin{figure}
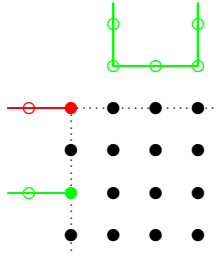

\[\unit=16pt
 \dotline\dl0{-.5}03\dl03{3.5}3\thnline
 \rd\pt03\pe{-1}{3}\dl{-1.5}{3}{0}{3}
 \bk\pt13\pt23\pt33
 \pt02\pt12\pt22\pt32
 \pt11\pt21\pt31
 \pt00\pt10\pt20\pt30
 \g\pt01\pe{-1}{1}\pe14\pe24\pe34\pe15\pe35
 \dl{-1.5}{1}{0}{1}\dl{1}{5.5}{1}{4}\dl{1}{4}{3}{4}\dl{3}{4}{3}{5.5}
 \hskip3\unit
\]
\caption{A problem at a corner in $\Z^3$. The red path enters the cube at the corner, but cannot
 enter the interior without becoming adjacent to the green path. The red path can
 however leave the corner via the 3rd dimension. The green path approaches the top of
 the cube before leaving and returning to the cube on the left.}\label{f:2}
\end{figure}

\section{Closing the gap}\label{sec_closing}

Since the particular combinatorial statement from which Aizenman and Grimmett~\cite{AG}
deduce Claim~\ref{cl2} is false, we seek a replacement.
We state this as a conjecture in general, since we can only prove certain cases.
In this conjecture, we may use either the $\ell_1$ or the $\ell_\infty$ distance; either
form of the conjecture immediately implies the other.

\begin{conjecture}\label{cpath}
Let $d\ge 2$ and $m\ge 1$. Then there is an $r=r(d,m)>0$ with the following property.
Let $P_R$ and $P_G$
be induced paths in $\Z^d$, one red and one green, each starting
outside $B_{r+2}$ and ending at a neighbour of $0$, with no red
vertex adjacent to any green vertex. Then we may modify the paths inside $B_r$
so that one ends at $-e_m$, the other ends at $e_m$, neither contains any other vertices
of $B_m$, and still no red vertex is adjacent to any green vertex.
\end{conjecture}

Less colourfully, the conjecture is equivalent to the assertion that if we have a single
induced path $P$ in $\Z^d$ through the origin,
joining two vertices outside $B_{r+2}$,
then we may modify $P$ inside $B_r$ to obtain an induced path $P'$
with the same endpoints so that $P'\cap B_m$ consists of
the line-segment joining $-e_m$ to $e_m$.

Note that the enhancement $\cE_0$ does not appear in Conjecture~\ref{cpath}, which is purely
a statement about paths in $\Z^d$; as in the original argument of Aizenman and Grimmett~\cite{AG},
the enhancement is only relevant in the proof of Lemma~\ref{inside}. The situation is similar
for site or bond percolation on other lattices.

Before turning to the proof of special cases of Conjecture~\ref{cpath}, let us show
that it implies Claim~\ref{cl2} and hence Conjecture~\ref{cmain} for site
percolation on $\Z^d$. (We discuss other percolation models briefly later.)

\begin{proof}[Proof that Conjecture~\ref{cpath} implies Conjecture~\ref{cmain}]
Fix $d\ge 2$, and assume that Conjecture~\ref{cpath} holds for this $d$.
In the light of the probabilistic argument of Aizenman and Grimmett~\cite{AG} and
the discussion in the previous section, it suffices to prove Claim~\ref{cl2}.

In the bulk of the argument we are about to give, it makes no difference whether
we work with the $\ell_1$ or $\ell_\infty$ metric; the exception is near the end
of the proof, when we consider only the $\ell_1$ metric. Formally, we consider
the $\ell_1$ metric throughout (in particular in defining $\cR_L$), but we only
indicate this choice once it becomes relevant.

Let $m$ be as given by Lemma~\ref{inside}, and let $r=r(m,d)$ be as in Conjecture~\ref{cpath}.
Set $R=3r+100$ and $L_0=100R$, say. Recall that $\cR_L$ is the event that the enhanced configuration
contains a path from $0$ to $S_L$.

Suppose that $L\ge L_0$, that $v$ is $p$-pivotal for the event $\cR_L$ in the state $(\omega,\alpha)$,
and that $\alpha\cap B_R(v)=0$. For the moment, suppose also that
\begin{equation}\label{notnear}
 r+10 < \|v\| < L -(r+10).
\end{equation}
Replacing $\omega$ by $\omega\cup\{v\}$ if necessary, we may assume that
$\cE(\omega,\alpha)$
contains a path from $0$ to $S_L$ but $\cE(\omega\setminus\{v\},\alpha)$ does not.
Delete, one-by-one, any sites $w\in \omega\cap B_{r+2}(v)$ not required for the existence
of a path from $0$ to $S_L$. Since $\alpha\cap B_{2r+2}(v)=\emptyset$,
and the range of the enhancement is at most $m-1<r$,
these sites play no role in any active enhancement, so we obtain a configuration $\omega'$
with the following properties:

(i) $\cE(\omega',\alpha)$ contains a path from $0$ to $S_L$

(ii) for any $w\in \omega'\cap B_{r+2}(v)$, the configuration
$\cE(\omega',\alpha)\setminus\{w\}$ contains no such path.

Since $v$ was $p$-pivotal we cannot delete it, so $v\in\omega'$.

Let $P=v_0v_1\cdots v_n$ be a shortest path from $0$ to $S_L$ in $\cE(\omega',\alpha)$. Then
$P$ is an \emph{induced} path (since any shortest path is). Moreover, (ii)
implies that $\omega'\cap B_{r+2}(v)=\cE(\omega',\alpha)\cap B_{r+2}(v)$
contains no site not on $P$, so (since $P\subseteq\cE(\omega',\alpha)$ and no
active enhancement affects $B_{r+2}(v)$)
we have
\begin{equation}\label{odr2}
 \omega'\cap B_{r+2}(v) = P\cap B_{r+2}(v).
\end{equation}

At this point the idea is to make the change inside $B_r(v)$ suggested by Conjecture~\ref{cpath},
and drop in the finite configuration given by Lemma~\ref{inside}, to obtain
a state with $v$ $s$-pivotal for the event $\cR_L$. We must be slightly
careful, because the configuration outside $B_{r+2}(v)$ need not consist only of $P$.
However, there is no real problem.

To spell things out, let $0$ and $z\in S_L$ be the ends of $P$.
Noting that $v$ is a vertex of the induced path $P$,
take as red and green paths the parts of $P$ obtained by deleting $v$.
Since $P$ is a shortest path the only point of $P$ on $S_L$ is $z$.
Hence, writing $\omega^P=P\setminus B_r(v)$ for the part of $P$ outside
$B_r(v)$,
Conjecture~\ref{cpath}
and Lemma~\ref{inside} together guarantee the existence of a
configuration $\omega_1\subset B_r(v)$ such that

(a) $\omega^P\cup \omega_1$ does not contain a path from $0$ to $S_L$ but

(b) $\omega^P\cup \omega_1 \cup X$ does,

\noindent
where $X=\cE_0(\omega_1-v)+v$. Let
\[
 \omega''= (\omega'\setminus B_r(v))\cup \omega_1.
\]
We claim that in the state $(\omega'',\alpha)$,
the site $v$ is $s$-pivotal. This will establish the required conclusion of Claim~\ref{cl2},
under the additional assumption \eqref{notnear}.

Since no enhancement within distance $2r$ of $v$ is active, inside $B_r(v)$
the enhanced configurations $\cE(\omega'',\alpha)$ and $\cE(\omega'',\alpha\cup\{v\})$
agree with $\omega_1$ and $\omega_1\cup X$, respectively.
Let
\begin{equation}\label{ostar}
 \omega^*=\cE(\omega'',\alpha)\setminus B_r(v) =\cE(\omega'',\alpha\cup\{v\})\setminus B_r(v)
 = \cE(\omega',\alpha)\setminus B_r(v).
\end{equation}
Then we must show exactly that

(a') $\omega^*\cup \omega_1$ does not contain a path from $0$ to $S_L$ but

(b') $\omega^*\cup\omega_1\cup X$ does.

\noindent
Note that (a') and (b') differ from (a) and (b)
in that $\omega^P$ is replaced by $\omega^*$.

A configuration $\omega^0$ outside $B_r(v)$ induces an equivalence relation $\sim$ on
$\omega^0\cap S_{r+1}(v)$, with two sites related if and only if $\omega^0$ contains
a path joining them. For each class, we note whether or not the sites in that
class are connected (in $\omega^0$) firstly to $0$, and secondly to $S_L$. For $\omega^P$,
it is easy to describe this relation: as $P$ is an induced path, the only connections
possible in $\omega^P$ are along sections of $P$.
Write $P=O_0I_1O_1\cdots I_kO_k$ where $k\ge 1$, each $O_i$ is a sequence of one or
more vertices outside $B_r(v)$, and each $I_i$ is a sequence of one or more vertices
inside $B_r(v)$. Then $\omega^P$ consists of the union of the paths $O_i$,
so there is one equivalence class $C_i$ for each set $O_i\cap S_{r+1}(v)$, $0\le i\le k$,
with only $C_0$ joined to $0$ and only $C_k$ joined to $S_L$.

We claim that $\omega^*$ induces the same equivalence relation and additional data.
First, recalling \eqref{odr2}, \eqref{ostar}, and that $\alpha$ contains
no active enhancement affecting sites in $B_{r+2}(v)$,
\[
 \omega^*\cap B_{r+2}(v)=(\omega'\setminus B_r(v))\cap B_{r+2}(v) = \omega^P\cap B_{r+2}(v).
\]
Thus $\omega^*$ and $\omega^P$ coincide in $S_{r+1}(v)$ and the two equivalence relations
have the same underlying set. Next, since $P$ is a path in $\cE(\omega',\alpha)$,
we have $\omega^P\subseteq \omega^*$, so two vertices of $S_{r+1}(v)$ connected
in $\omega^P$ are connected in $\omega^*$. Suppose for a contradiction that for some $i<j$
there is a path in $\omega^*$ joining the classes $C_i$ and $C_j$ defined above.
Then $I_j$ becomes redundant: in
$\omega^*\cup ((P\setminus I_j)\cap B_r) = \cE(\omega',\alpha)\setminus I_j$
there is a path from $0$ to $S_L$. But this contradicts the minimality condition (ii).
We obtain a similar contradiction if $\omega^*$ contains a path from $0$ to $C_i$, $i\ne 0$,
or from $S_L$ to $C_i$, $i\ne k$. Hence, in terms of connecting points of $S_{r+1}(v)$
to each other and/or to $0$ or to $S_L$, the configurations $\omega^*$ and $\omega^P$
are equivalent. Thus (a) and (b) imply (a') and (b').

\medskip

Recalling \eqref{notnear}, it remains only to consider the cases $v\in B_{r+10}(0)$
and $v\in B_L(0)\setminus B_{L-r-10}(0)$. We remind the reader that $(\omega,\alpha)$
is a state in which $v$ is $p$-pivotal for $\cR_L$, with $\alpha\cap B_R(v)=0$.

The case $v\in B_{r+10}(0)$ is easily handled: in this case in the state obtained
by adding all points of $B_{r+20}(0)$ to $\omega$ there is a path from $0$ to $S_L$.
Now delete all points of $S_{r+19}(0)$ except for $w=e_{r+19}$, say.
The point $w$ is now
$p$-pivotal for the existence of a path from $0$ to $S_L$, so we may apply the argument
above with $w$ in place of $v$.

If $v$ is close to $S_L=\Sone_L(0)$ we need to work a little harder, but the situation is not too bad
since we need only consider one path.
Here we shall consider only the $\ell_1$ metric, so $L-(r+10)\le \|v\|_1\le L$.
(Clearly no point $v$ with $\|v\|_1>L$ can be $p$-pivotal when $\alpha\cap \Bone_r(v)=\emptyset$).
Choose a point $w$ so that $\|w\|_1=L-(r+20)$ and $\|v-w\|\le r+20$.
One by one, delete points in $\Bone_{r+30}(w)$ from $\omega$ that are not required to join
$0$ and $\Sone_L$. As above, eventually we are left with a set $\omega'$ such that 
$\omega'\cap\Bone_{r+30}(w)$ consists of the intersection of $\Bone_{r+30}(w)$
with an induced path $P$ joining $0$ to $\Sone_L$ via~$v$. Let $u$ be the first point
of $P$ that lies in $\Bone_{r+30}(w)$. As $u$ occurs before $v$ on $P$,
we have $\|u\|_1<L$. Now remove from $\omega$ all
points in $\Bone_{r+30}(w)$ except for~$u$. Suppose we can construct an induced path
$P'$ from $u$ to $\Sone_L$ passing through $w$ such that $P'$ is contained in $\Bone_{r+29}(w)$
except for the single vertex $u\in \Sone_{r+30}(w)$.
Then, noting that $\alpha\cap\Bone_{2r+30}(w)\subset \alpha\cap\Bone_{3r+50}(v)=\emptyset$,
the site $w$ is $p$-pivotal for the event $\cR_L$
in the state $((\omega\setminus\Bone_{r+30}(w))\cup P',\alpha)$.
Since $\|w\|_1<L-(r+10)$,
we are then done by the above argument. It thus remains only to construct~$P'$.

Let $u=(u_1,\dots,u_d)$ and $w=(w_1,\dots,w_d)$. We construct the
segment of the path $P'$ from $u$ to $w$ by changing each coordinate monotonically,
so that the distance to $w$ reduces by 1 at each step. To ensure that the
path stays inside $\Bone_{L-1}$, we take all the steps which reduce the
absolute value of some coordinate first, then take all the remaining
steps that increase the absolute value of some coordinate. Suppose
the last step on this path is from $w-e$ to $w$, where
$e=(0,\dots,\pm1,\dots,0)$ and the $\pm1$ lies in the $i$th coordinate, say.
We then take one further step to $w+e$, and then all subsequent steps
increase the absolute value of some coordinate other than the~$i$th.
After at most $r+21$ steps we hit $\Sone_L$ at some point in the
interior of $\Bone_{r+30}(w)$. Moreover the resulting path
$P'$ is induced. Indeed, all points of $P'$ after $w$ differ from all
points of $P'$ before $w$ by at least 2 in coordinate~$i$, while
the segments to and from $w$ change the distance to $w$ monotonically.
\end{proof}

\section{The proof of Lemma~\ref{inside}}\label{sec_inside}

For completeness we give a proof of Lemma~\ref{inside}, even
though it is essentially trivial. Here we consider the $\ell_\infty$ ball.

\begin{proof}[Proof of Lemma~\ref{inside}]
Suppose that the essential enhancement $\cE_0$ has range $r$, with respect to the $\ell_\infty$ metric.
By the definition of essential enhancement, there is a configuration
$\omega'$ such that $\omega'$ contains no two-way infinite path,
but $\cE(\omega',\{0\})=\omega'\cup\cE_0(\omega')$ does.
Let $X=\cE_0(\omega')$ so, since the enhancement has range $r$,
$X\subseteq \Binf_r=\Binf_r(0)$.
Let $u$ and $v$ be two points (i.e., vertices/sites) in $\omega'\cap\Sinf_{r+1}$
in the same component of $(\omega'\cup X)\cap\Binf_{r+1}$, but
in different components of $\omega'\cap\Binf_{r+1}$. Such points exist
since $\omega'$ and $\omega'\cup X$ differ only inside $\Binf_r$ but induce different connectivity
relations on the points outside $\Binf_r$.

As $u\in\Sinf_{r+1}$ there must be a coordinate of $u$ that
is $\pm(r+1)$. Let $u_1,u_2$ be the points obtained by increasing
the magnitude of one such coordinate by one and two steps respectively, so
that $uu_1u_2$ forms a straight path with $u_1\in \Sinf_{r+2}$, $u_2\in\Sinf_{r+3}$.
Define $v_1,v_2$ similarly. Note that none of $v,v_1,v_2$ is adjacent to any of
$u,u_1,u_2$. Writing $e_k$ for $(k,0,\dots)$, our aim now is to construct a path
$P_u$ from $u_2$ to one of $\pm e_{m-1}$, for some $m\ge r+4$, and a path $P_v$ from $v_2$ to the other
of $\pm e_{m-1}$, so that $P_u\cup P_v\subseteq \Binf_{m-1}\setminus\Binf_{r+2}$ and no vertex of $P_u$
is adjacent to any vertex of $P_v$. Since the only vertices (sites) present in $\Sinf_{r+2}$
are $u_1$ and $v_1$, adding the vertices $\pm e_m$ will then give the required configuration
$\omega$.

Write $u_2=(a_1,a_2,\dots,a_d)$ and $v_2=(b_1,b_2,\dots,b_d)$.
By interchanging $u$ and $v$ and/or reflecting in the hyperplane $x_1=0$,
we may assume without loss of generality that $a_1\le b_1$, and that if $a_1=b_1=\pm (r+3)$
then $a_1=b_1=r+3$.

Consider first the case $a_1<r+3$. Starting from $u_2$, construct $P_u$
by first taking successive steps reducing $a_1$ to $-(r+3)$,
reaching the `left-hand' face of the cube $\Sinf_{r+3}$. Then continue to $-e_{r+3}$ within this face,
successively reducing the magnitude of each of the other
coordinates in turn to 0. Similarly, starting from $v_2$ construct $P_v$ by
first increasing $b_1$ (which is not equal to $-(r+3)$) to $r+3$,
i.e., moving to the right-hand face if not already in it.
Then successively reduce the magnitude
of each of the other coordinates in turn to~0. These paths have the required property
with $m=r+4$.

Finally, consider the case $a_1=b_1=r+3$, when both $u_2$ and $v_2$ are in the
right-hand face. Since $u_2$ and $v_2$ differ in some coordinate,
we may assume without loss of generality that $a_2<b_2$. Starting from $u_2$, construct
what will be the first part of $P_u$ by decreasing
$a_2$ until it reaches $-(r+5)$. Similarly, from $v_2$ increase
$b_2$ until it reaches $r+5$. Let $u'$ and $v'$ be the end vertices of the paths
constructed so far. Then $u'$ and $v'$ are in the `top' and `bottom' faces of $\Sinf_{r+5}$.
Since neither is in the right-hand face of $\Sinf_{r+5}$ we
may continue as in the previous case with $u'$, $v'$ in place of $u_2$, $v_2$ and $r+3$ in place of~$r$.
\end{proof}

\section{The proof for $\Z^2$ and $\Z^3$}\label{sec_Z2Z3}

In this section we prove Conjecture~\ref{cpath} for $d=2$ and for $d=3$,
thus proving Theorem~\ref{thmain}. To avoid the problems described
at the end of Section~\ref{sec_AG}, we work with the $\ell_1$-ball.
Then it can be seen that the only difficult cases occur when one
of the given coloured paths first meets $\Bone_r$ at a corner.

\subsection{The square lattice}

\begin{proof}[Proof of Conjecture~\ref{cpath} for $d=2$]
We work throughout with the $\ell_1$ metric. We shall prove
the result with $r=r(m,2)=m+c$ for some constant $c$. We assume
without loss of generality that $m\ge 50$, say. It will be convenient
to consider the octagon $O_r$ formed by removing the corner points
$(\pm r,0)$, $(0,\pm r)$ from $\Bone_r=\Bone_r(0)$.

\begin{figure}
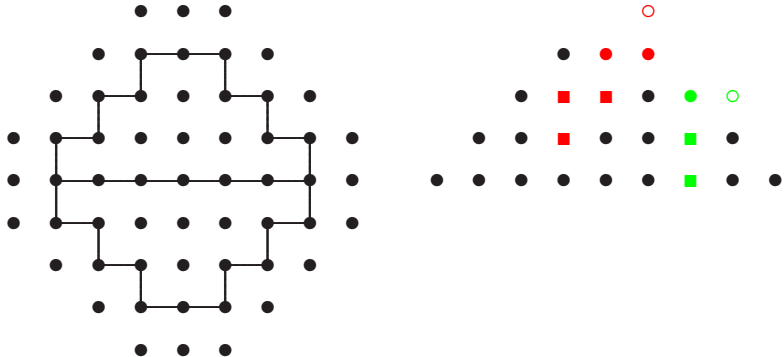

\[\unit=16pt
 \pt{-1}{4}\pt{0}{4}\pt{1}{4}
 \pt{-2}{3}\pt{-1}{3}\pt{0}{3}\pt{1}{3}\pt{2}{3}
 \pt{-3}{2}\pt{-2}{2}\pt{-1}{2}\pt{0}{2}\pt{1}{2}\pt{2}{2}\pt{3}{2}
 \pt{-4}{1}\pt{-3}{1}\pt{-2}{1}\pt{-1}{1}\pt{0}{1}\pt{1}{1}\pt{2}{1}\pt{3}{1}\pt{4}{1}
 \pt{-4}{0}\pt{-3}{0}\pt{-2}{0}\pt{-1}{0}\pt{0}{0}\pt{1}{0}\pt{2}{0}\pt{3}{0}\pt{4}{0}
 \pt{-4}{-1}\pt{-3}{-1}\pt{-2}{-1}\pt{-1}{-1}\pt{0}{-1}\pt{1}{-1}\pt{2}{-1}\pt{3}{-1}\pt{4}{-1}
 \pt{-3}{-2}\pt{-2}{-2}\pt{-1}{-2}\pt{0}{-2}\pt{1}{-2}\pt{2}{-2}\pt{3}{-2}
 \pt{-2}{-3}\pt{-1}{-3}\pt{0}{-3}\pt{1}{-3}\pt{2}{-3}
 \pt{-1}{-4}\pt{0}{-4}\pt{1}{-4}
 \thnline\dl{-1}{3}{1}{3}\dl{1}{3}{1}{2}\dl{1}{2}{2}{2}\dl{2}{2}{2}{1}\dl{2}{1}{3}{1}
 \dl{3}{1}{3}{-1}\dl{3}{-1}{2}{-1}\dl{2}{-1}{2}{-2}\dl{2}{-2}{1}{-2}\dl{1}{-2}{1}{-3}
 \dl{1}{-3}{-1}{-3}\dl{-1}{-3}{-1}{-2}\dl{-1}{-2}{-2}{-2}\dl{-2}{-2}{-2}{-1}
 \dl{-2}{-1}{-3}{-1}\dl{-3}{-1}{-3}{1}\dl{-3}{1}{-2}{1}\dl{-2}{1}{-2}{2}
 \dl{-2}{2}{-1}{2}\dl{-1}{2}{-1}{3}
 \dl{-3}{0}{3}{0}
 \hskip10\unit
 \pt{-1}{3}
 \pt{-2}{2}\pt{1}{2}
 \pt{-3}{1}\pt{0}{1}\pt{1}{1}\pt{3}{1}
 \pt{-4}{0}\pt{-3}{0}\pt{-1}{0}\pt{0}{0}\pt{1}{0}\pt{3}{0}\pt{4}{0}
 \pt{-2}{1}\pt{-2}{0}
 \g\pe{3}{2}\pt{2}{2}\px{2}{1}\px{2}{0}
 \rd\pe{1}{4}\pt{1}{3}\pt{0}{3}\px{0}{2}\px{-1}{2}\px{-1}{1}
\]
\caption{(a) The octagon $O_r$ (dots), and cycle $C$ in $O_{r-1}\setminus O_{r-3}$ along
with the line joining $\pm(r-2,0)$. (b) Red and green paths first entering $O_r$
at $(1,r-1)$ and $(2,r-2)$. The circles show vertices on the original paths (hollow ones
outside $O_r$); the squares show
a way of continuing inside $O_{r-1}$ to first reach $O_{r-2}$ at points (the final coloured
squares) that are not diagonally adjacent.}\label{f:3}
\end{figure}

We are given red and green paths $P_R$ and $P_G$
starting outside $\Bone_{r+2}$ and ending near the origin.
Truncate each path at its first vertex in $O_r$,
obtaining shortened paths $P_R'=R_1\cdots R_a$ and $P_G'=G_1\cdots G_b$
with $R_a$ and $G_b$ in $O_r$ and all other vertices of $P_R'$ and $P_G'$
outside $O_r$. Recall that no red vertex
is adjacent to any green vertex.
We aim to add points in the interior $O_{r-1}$ of $O_r$ so as to join
$R_a$ and $G_b$ by an induced path whose intersection
with $\Bone_m$ is $\{(i,0):i=-m,\dots,m\}$. Equivalently, we aim to add red and green points
in $(O_{r-1}\setminus\Bone_m)\cup \{\pm e_m\}$
so that $R_a$ is connected to one of $\pm e_m$ by a red path,
$G_b$ is connected to the other of $\pm e_m$ by a green path, and no red point is adjacent to
any green point.

Note that the `annulus' $O_{r-1}\setminus O_{r-3}$ contains a cycle $C$ that covers
all vertices of $O_{r-1}\setminus O_{r-2}$; see \Fg{3}(a).

We proceed in three steps. First suppose that $R_a$ and $G_b$
are far apart, say at distance at least $10$, although some much
smaller constant suffices.
Then we may join $R_a$ to a neighbour $u$ on $C$ and $G_b$ to a neighbour
$v$ on $C$, and continue around $C$ to the points $\pm e_{r-2}=(\pm(r-2),0)$. Since
$u$ and $v$ are far apart it is possible to do this without the paths
within $C$ getting close to each other. For example, if the $x$-coordinates of $u$ and $v$ differ
significantly, proceed to the left to $-e_{r-2}$ from the point with smaller $x$-coordinate,
and to the right to $e_{r-2}$ from the point with larger $x$-coordinate. If the $x$-coordinates
are close, then the $y$-coordinates differ significantly, have opposite signs, and neither is close to $0$,
so we may proceed clockwise from each of $u$ and $v$, joining one to $e_{r-2}$ and the other
to $-e_{r-2}$. Finally, connect $e_{r-2}$ and $-e_{r-2}$ with a straight line.

Next suppose that $R_a$ and $G_b$ are within distance $9$, but are not diagonally adjacent
points with one being a corner of $O_{r-1}$. Then we can find neighbours $u$ and $v$
of $R_a$ and $G_b$ on $C$ so that neither of $u,R_a$ is adjacent to either of $v,G_b$.
Then we can move
around $C$ from $u$ and $v$ in opposite directions, to reach points $u'$ and $v'$
that are far from each other. Now
we take two steps from each of these points into $O_{r-3}$ and apply the first case above
with $r$ replaced by $r-3$.

Finally, suppose that one of $R_a$, $G_b$ is a corner of $O_r$
and the other is diagonally adjacent to it, for example
$R_a=(1,r-1)$ and $G_b=(2,r-2)$.
The original red path must have come from somewhere, and in particular
from a site $R_{a-1}$ outside $O_r$ and not adjacent to any green vertex;
the only possibility is $R_{a-1}=(1,r)$. The red path must also have continued somewhere,
and the only possibility is $R_{a+1}=(0,r-1)$. But then we can extend the red
and green paths into the interior of $O_r$ as shown in \Fg{3}(b).
These paths enter $O_{r-2}$ at points that are not diagonally adjacent, so we may apply
the previous case.
\end{proof}

\subsection{The cubic lattice}

We now show that we can continue non-adjacent coloured paths
as in Conjecture~\ref{cpath} into an $\ell_1$-ball
in $\Z^3$. The proof is significantly more complex than for the 2-dimensional
case, so we start with a couple of preparatory lemmas. These lemmas spell out things
that are `obvious', but in context we feel it is appropriate to give more detail
than usual. Throughout we work with the $\ell_1$ metric,
so $B_r=\Bone_r=\Bone_r(0)$
and $S_r=\Sone_r=\Sone_r(0)$.

\begin{lemma}\label{3far}
Let $d\ge 2$ and $m\ge 1$. Suppose that $r\ge m+10$.
Let $u=(u_1,\ldots,u_d)$ and $v=(v_1,\ldots,v_d)$ be points in $S_r$
such that, for some coordinate $i$, there is an integer $s$ with $|s|\le r-5$
and $u_i>s>v_i$ or $v_i>s>u_i$.
Then we may construct a red path starting at $u$ and a green path starting at $v$
with the following properties: one path ends at $-e_m$ and the other ends at $e_m$,
apart from their endvertices the paths lie entirely in $B_{r-1}\setminus B_m$,
and no red vertex is adjacent to any green vertex.
\end{lemma}
\begin{proof}
Swapping $u$ and $v$ if necessary, we may assume that
$u_i>s>v_i$. Let $A_1=(S_{r-1}\cup S_{r-2})\cap \{x:x_i>s\}$ and
$A_2=(S_{r-1}\cup S_{r-2})\cap \{x:x_i<s\}$
be the two parts of the `annulus' $S_{r-1}\cup S_{r-2}$ obtained by deleting all points
in the hyperplane $x_i=s$.
Then, since $|s|\le r-3$, each $A_i$ induces a connected subgraph of $\Z^d$,
and $u$ has a neighbour in $A_1$ and $v$ has a neighbour in $A_2$.
Thus we can join $u$ by a red path in $A_1$ to
the point $w=(0,\cdots,r-2,\cdots,0)$, where the non-zero coordinate
is the $i$th. Similarly, we may join $v$ by a green path in $A_2$ to $-w$.
So far, red and green points are on opposite sides of
$\{x:x_i=s\}$, so no red point is adjacent to any green point.
Since $|s|\le r-5$, we may extend the red and green paths by appending two steps towards the origin
to each while remaining on opposite sides of $\{x:x_i=s\}$. The paths now end at opposite corners
$w'$ and $-w'$ of $B_{r-4}$, and each has exactly one vertex in $S_{r-3}$.

Continuing inside $B_{r-4}$ no new vertices added will be adjacent to any existing vertices
other than $w'$ and $-w'$. If $i=1$ we are done: simply continue in a straight line along
the $x$-axis from $w'=e_{r-4}$ to $e_m$ and from $-w'$ to $-e_m$. Otherwise, suppose without loss
of generality that $i=2$. Then we may join $w'$ to $e_{r-4}$ within the set
$\{(x,y,0,\cdots): x\ge 0, y\ge 0, r-5\le x+y\le r-4\}$ and $-w'$ to $-e_{r-4}$ similarly, reflecting
in the origin. Finally, continue in a straight line along the $x$-axis as before.
\end{proof}

So far, we wrote the argument for general $d$ since there were no extra complications. From
now on we consider only $d=3$, although some (but not all!) further parts of our argument
extend easily to higher dimensions.

Note that any points of $S_r\subset \Z^3$ at distance at least $30$, say, automatically
satisfy the assumptions of Lemma~\ref{3far}, since they differ by at least $10$ in some coordinate.
Also, while the statement of Conjecture~\ref{cpath} distinguishes one particular
coordinate (the first), the assumptions of Lemma~\ref{3far} do not. Since
we shall use Lemma~\ref{3far} to prove Conjecture~\ref{cpath} for $d=3$, this
means that from now on we can treat all coordinates as equivalent.

\begin{lemma}\label{3step}
Let $r\ge m+20$ and let $u$ and $v$ be two points in $S_r$ that do not satisfy the assumptions
of Lemma~\ref{3far}. Unless one of $u$, $v$ is a corner of $S_r$ and $\none{u-v}=2$, then
either

(a) we can find neighbours $u'$ and $v'$ of $u$ and $v$ in $S_{r-1}$ so that
there is no edge from $\{u,u'\}$ to $\{v,v'\}$
and $\none{u'-v'}>\none{u-v}$, or

(b) there are paths $P_u$ from $u$ to some $u'\in S_{r-3}$
and $P_v$ from $v$ to some $v'\in S_{r-3}$ such that
$u'$ and $v'$ satisfy the assumptions of Lemma~\ref{3far} with $r$ replaced by $r-3$,
$P_u$ and $P_v$ are contained in $S_{r-1}\cup S_{r-2}$ apart from their endpoints,
and $P_u$ and $P_v$ are at distance at least $2$ from each other.
\end{lemma}

In checking the details of this and the next proof, it is perhaps helpful to bear in
mind that if $u$ and $v$ are adjacent, then $\none{u}$ and $\none{v}$ differ by exactly $1$.

\begin{proof}
Let $u=(a_1,a_2,a_3)$ and $v=(b_1,b_2,b_3)$. Since $u$ and $v$ do not
satisfy the assumptions of Lemma~\ref{3far}, there is no coordinate in which
one is positive and the other negative. Thus
we may assume without loss of generality that $a_i\ge 0$ and $b_i\ge 0$ for all $i$.

Suppose first that neither $u$ nor $v$ is at a corner of $S_r$.
Then at least two of the $a_i$ and at least two of the $b_i$ are strictly
positive, so there is some coordinate in which both are positive,
and we may assume that $0<a_1\le b_1$.

Suppose that (i) $0<b_2\le a_2$. Then we may take $u'=(a_1-1,a_2,a_3)$ and $v'=(b_1,b_2-1,b_3)$
and we are done. Similarly, if (ii) $0<b_3\le a_3$, take the same $u'$ and $v'=(b_1,b_2,b_3-1)$.
Thus we may assume that neither (i) nor (ii) holds.

Since $\none{u}=\none{v}=r$ and $b_1\ge a_1$, without loss of generality we have $b_3<a_3$.
Now we must have $b_3=0$, otherwise (ii) holds. Thus $b_2>0$ (recall than at most one $b_i$
may be zero), and thus $b_2>a_2$ as otherwise (i) holds. Now $a_3>b_3=0$, so $a_3\ge 1$.
If $a_3\ge 2$, then the conditions of Lemma~\ref{3far} are satisfied with $i=3$ and $s=1$.
Thus $a_3=1$. Since $\none{u}=\none{v}$, $b_1\ge a_1$ and $b_2>a_2$, it follows that
$b_1=a_1$ and $b_2=a_2+1$, so $u=(x,y,1)$ and $v=(x,y+1,0)$ for some $x>0$ and $y\ge 0$.

If $y>0$ then we may take $u'=(x,y-1,1)$ and $v'=(x-1,y+1,0)$. Otherwise
$u=(r-1,0,1)$ and $v=(r-1,1,0)$. We build a path $P_u$ starting from $u$
going via $(r-2,0,1)$, $(r-3,0,1)$, $(r-3,0,2)$ and $(r-4,0,2)$ to $u'=(r-5,0,2)$,
and construct the analogous path $P_v$ (with $y$- and $z$-coordinates swapped) from $v$ to $v'=(r-5,2,0)$.
These paths satisfy the second alternative (b) in the conclusion of the lemma.

It remains to handle case where one of $u$ and $v$ is a corner of $S_r$. Suppose
without loss of generality that $u=(0,0,r)$.
Then (by the assumption that in this case $u$ and $v$ are at distance
more than $2$) $v=(b_1,b_2,b_3)$ with $b_3\le r-2$.
Note also that $b_3\ge r-5$, as otherwise the conditions of Lemma~\ref{3far} hold;
hence, crudely, $b_3\ge 3$.
We may assume without loss of generality that $b_2>0$ and $b_1\ge 0$, say.
Thus $b_1+b_2=r-b_3\ge 2$.

Build a path from $u$ to $(0,0,r-1)$, $(0,0,r-2)$, $(0,-1,r-2)$
$(0,-1,r-3)$ and $u'=(0,-1,r-4)$, and a path from $v$ to $v'$
by taking $3$ steps in the negative $z$ direction.
These paths satisfy the second alternative in the conclusion of the lemma.
\end{proof}

We now turn to the main result of this section.

\begin{proof}[Proof of Conjecture~\ref{cpath} for $d=3$]
We shall prove the result with $r=r(3,m)=m+c$ where $c$ is some absolute constant that
we shall not optimise.
Recall that we are given red and green paths $P_R$ and $P_G$
starting outside $B_{r+2}$ and ending at neighbours of the origin, and we must modify the paths
in a certain way inside $B_r$. It will be convenient to shift the index: replacing $r$ by $r+1$
our paths start outside $B_{r+3}$, and we are allowed to modify them inside $B_{r+1}$.

In the following argument, a key role will be played by the points
of the paths $P_R$ and $P_G$ that lie in $S_r$ and come before the first time the 
relevant path
enters $B_{r-1}$. To avoid double subscripts, we shall label these
points as $R_1,\ldots,R_s$ in order along $P_R$ and $G_1,\ldots,G_t$
along $P_G$. Thus $R_i$ is the $a_i$th point of $P_R$ for some $a_1<a_2<\cdots<a_s$,
the $(a_s+1)$st point of $P_R$ is in $S_{r-1}$, and all other points of $P_R$
before the $(a_s+1)$st are outside $B_r$.

Let (for the moment) $u=R_1$ and $v=G_1$. If these points satisfy the assumptions
of Lemma~\ref{3far}, then we may delete all points of $P_R$ and $P_G$ in $S_r$ other
than $u$ and $v$, and use Lemma~\ref{3far} to continue the paths inside $B_r$ in the required manner.
Indeed, all new vertices added are inside $B_{r-1}$, and so have no neighbours on what
remains of $P_R$ and $P_G$ except $u$ and $v$.
Similarly, if $u$ and $v$ satisfy the assumptions of Lemma~\ref{3step} then we may apply
that lemma a bounded number of times to arrive at a situation to which Lemma~\ref{3far}
applies, with $r$ reduced by at most $40$, say, and so $r\ge m+10$ still, as required.
We may thus assume that neither Lemma~\ref{3far} nor Lemma~\ref{3step} applies,
so one of $R_1$ and $G_1$ is a corner of $S_r$ and $\none{R_1-G_1}=2$.

Let us say that the quadruple $(i,j,u',v')$ is \emph{good} if the following conditions hold:

 (i) $i\le s$ and $j\le t$, so $R_i$ and $G_j$ are defined,

 (ii) $u'\in S_{r-1}$ is a neighbour of $R_i$ and $v'\in S_{r-1}$ is a neighbour of $G_j$,

 (iii) $u'\ne v'$
and either $\none{u'-v'}>2$ or neither $u'$ nor $v'$ is a corner of $S_{r-1}$, and

(iv) $u'$ is not adjacent to any $G_k$, $k\le j$, and
$v'$ is not adjacent to any $R_k$, $k\le i$.

If such a good quadruple exists then we may delete all points of the original
paths in $S_r$ other than the points $R_k$, $k\le i$, and $G_k$, $k\le j$,
continue from $u=R_i$ to $u'$ and from $v=G_j$ to $v'$, and then apply Lemmas~\ref{3far}
and~\ref{3step} as above. Hence we may assume that no good quadruple exists.

\begin{claim}\label{cldiag}
Under the assumptions above, either
$R_1$ is a corner of $B_r$ and $R_1,G_1,R_2,G_2,\dots,R_{r/4},G_{r/4}$
are successive vertices along one edge of the octahedron $B_r$,
or the same situation holds with red and green swapped.
\end{claim}

We prove the claim by induction. Recalling that one of $R_1$ and $G_1$ is a corner of $B_r$
and the other is at distance $2$ from it, suppose that we have the pattern in the claim,
from the corner $(0,0,r)$ down to the point $(n,0,r-n)$, $n\ge 1$. We shall show by induction that
if $n<r/2$, say, then the pattern continues one more step.
Swapping the colours if necessary, suppose that $(n,0,r-n)$ is red; in particular,
$(n,0,r-n)=R_k$ where $k=\floor{n/2}+1$, and $(n-1,0,r-n+1)=G_\ell$ where $\ell=\floor{(n-1)/2}+1$.
Our aim is to show that $G_{\ell+1}$ is defined and is equal to $(n+1,0,r-n-1)$.

The only neighbours of $G_\ell$ in $B_{r-1}$ are the points
$(n-1,0,r-n)$ and (if $n\ge 2$) $(n-2,0,r-n+1)$. The former has a red neighbour, namely
$R_k$. If $n\ge 2$, then $(n-2,0,r-n+1)$ also has a red neighbour, namely $R_{k-1}=(n-2,0,r-n+2)$.
Hence the successor of $G_\ell$ on the green path is outside $B_r$,
and $G_{\ell+1}$ (the next time the path returns to $B_r$) is defined.

Suppose that we can find distinct neighbours $u'$ of $u=R_k$ and $v'$ of $v=G_{\ell+1}$ in $S_{r-1}$
such that

 (*) $u'$ is not adjacent to any of $G_1,\ldots,G_{\ell+1}$ and $v'$
is not adjacent to any of $R_1,\ldots,R_k$.

\noindent
Then the quadruple $(k,\ell+1,u',v')$ is good. Indeed, the only remaining condition
to check in the definition of a good quadruple is (iii). To verify this note that
neither $u'$ nor $v'$ is the corner $(0,0,r-1)$ (which is adjacent to $R_1$ and to $G_1$),
so either they are far apart or neither is a corner of $S_{r-1}$. The condition (*)
is very easy to check, since $A=\{R_1,\ldots,R_{k-1},G_1,\ldots,G_\ell\}
=\{(m,0,r-m), 0\le m\le n-1\}$. Moreover, the neighbourhood of $A$ in $S_{r-1}$ is
simply $B=\{(m,0,r-m-1), 0\le m\le n-1\}$. Hence, to satisfy (*) it suffices
to ensure that $u',v'\notin B$, that $u'\ne v'$, and that neither $u'v$ nor $v'u$ is an edge.

Consider choosing $u'=(n,0,r-n-1)$ as the neighbour of $u=R_k$.
If $v=G_{\ell+1}$ does not have a strictly positive $z$-coordinate, then there
is no problem: $v$ is far from all relevant vertices in $B_r$ and we may take any neighbour
$v'$ of $v$ in $S_{r-1}$.
Otherwise, consider taking $v'=v-(0,0,1)$. Since $v\notin A$, we have $v'\notin B$,
and these choices work unless either $u'v$ or $v'u$ is an edge, which
(bearing in mind that the $G_i$ and $R_j$ are distinct)
happens precisely
when $v=G_{\ell+1}$ is one of the points
$(n+1,0,r-n-1)$ or $(n,\pm 1, r-n-1)$. The first is what we are trying to establish (that the
pattern continues along the edge of $B_r$). So suppose without loss of generality
that $G_{\ell+1}=(n,1,r-n-1)$.

Now the only neighbours of $R_k=(n,0,r-n)$ in $S_{r-1}$ are $(n-1,0,r-n)$ and
$(n,0,r-n-1)$. Since each has a green neighbour ($G_\ell$ or $G_{\ell+1}$),
the red path cannot have entered $B_{r-1}$ after $R_k$, so $R_{k+1}$ is defined.
Now we attempt to enter $B_{r-1}$ from $u=R_{k+1}$ and $v=G_{\ell+1}=(n,1,r-n-1)$,
i.e., to find a good quadruple $(k+1,\ell+1,u',v')$.
As before, taking $u'=u-(0,0,1)$ and $v'=v-(0,0,1)=(n,1,r-n-2)$ works
unless $u'v$ or $v'u$ is an edge of $\Z^3$, since $u',v'\notin B$.
Now $v'u$ is an edge only if $u=(n+1,1,r-n-2)$ or $u=(n,2,r-n-2)$.
But in both cases the quadruple $(k+1,\ell+1,u',v''=(n-1,1,r-n-1))$
is good.
So we may assume that $u'v$ is an edge, and thus that $R_{k+1}=u=(n-1,1,r-n)$.
If $n\ge 2$ then there is no problem: keeping the same $v'$, take $u''=(n-2,1,r-n)$.

The only remaining case in the proof of Claim~\ref{cldiag} is the last one
above with $n=1$. In other words, $G_1=(0,0,r)$, $R_1=(1,0,r-1)$, $G_2=(1,1,r-2)$
and $R_2=(0,1,r-1)$. Now both neighbours of $R_2$ in $B_{r-1}$ have a green neighbour,
so the red path cannot have entered $B_{r-1}$ after $R_2$, and hence $R_3$ is defined.
Now $G_2$ has only three neighbours not known to have a red neighbour,
namely $(2,1,r-2)$, $(1,1,r-3)$ and $(1,2,r-2)$. Hence two of these must be green.
It follows that $R_3$ cannot be $(2,1,r-3)$ or $(1,2,r-3)$ (since each has
two neighbours in a set of three points containing at least two green points).
Now we claim that (unless $R_3$ has $z$-coordinate less than or equal to $0$,
but then it is far from $G_1$, $G_2$ and there is no problem) the quadruple
$(3,2,u',v')$ is good, where $u=R_3$, $u'=u-(0,0,1)$, $v=G_2=(1,1,r-2)$ and $v'=(1,1,r-3)$.
Indeed, the only neighbours of $v'$ (other than $v$) in $S_r$ are $(2,1,r-3)$ and $(1,2,r-3)$,
neither of which is $R_1$, $R_2$ or $R_3$, and for $u'$ to be a neighbour in $S_{r-1}$
of $G_1$ or $G_2$ we would have to have $u'=(0,0,r-1)$, $(1,0,r-2)$ or $(0,1,r-2)$
in which case $u$ is one of the points $G_1$, $R_1$, $R_2$, which is impossible
since $R_1$, $R_2$, $R_3$ and $G_1$ are distinct. This completes the proof of Claim~\ref{cldiag}.

\medskip
With Claim~\ref{cldiag} in hand, we continue with the proof of the
case $d=3$ of Conjecture~\ref{cpath}.
As noted near the start of the proof, we may assume there is no good quadruple. Hence,
by Claim~\ref{cldiag}, we may assume that $R_1=(0,0,r)$, $G_1=(1,0,r-1)$,
$R_2=(2,0,r-2)$, $G_2=(3,0,r-3)$ and $R_3=(4,0,r-4)$.
Since red and green points cannot be adjacent, the points on
the green path before and after $G_1$ must be
$(1,\pm1,r-1)$; these are the only neighbours of $(1,0,r-1)$
that are not adjacent to a red point. Similarly the points
before and after $G_2$ are $(3,\pm1,r-3)$, and the points
adjacent to $R_2$ on the red path must be $(2,\pm1,r-2)$.
We shall assume without loss of generality that
$(2,1,r-2)$ precedes $R_2$ on the red path.
Since $(1,1,r-1)$ and $(3,1,r-3)$ are both green, the only neighbours
of $(2,1,r-2)$ without green neighbours are $(2,0,r-2)=R_2$ and
$(2,2,r-2)$; hence this last point is also red, and comes just before $(2,1,r-2)$ on the red path.

Truncate the red path $P_R$ at $(2,1,r-2)$, the point before $R_2$. (More precisely, to keep a bound on the range
of our modification, delete all points of $P_R\cap B_{r+1}$ that come after
$(2,1,r-2)$ along $P_R$.) Similarly, truncate the green path at $G_1=(1,0,r-1)$.
Now continue the green path to $u=(1,0,r-2)\in S_{r-1}$ and the red path
to $v=(2,1,r-3)\in S_r$ and then to $w=(2,1,r-4)\in S_{r-1}$. At this
point $u$ and $w$ are the only coloured vertices in $S_{r-1}$.
The coloured vertices in $S_r$ are precisely $R_1=(0,0,r)$, $G_1=(1,0,r-1)$
and $v$. Hence there is no red-green adjacency between $S_{r-1}$ and $S_r$.
The only newly coloured point in $S_r$ is $v$, which is red. Its neighbours
in $S_{r+1}$ are $(2,1,r-2)$ which was and is red, $(3,1,r-3)$
which was green but, from truncating the green path, is now uncoloured,
and $(2,2,r-3)$ which cannot have been green since $(2,2,r-2)$ was red.
Thus the red and green paths to $u$ and $w$ remain at distance $2$ and,
since neither $u$ nor $w$ is a corner of $B_{r-1}$, we can continue them inside $B_{r-1}$
by a case previously covered.
\end{proof}

\section{Further percolation models}\label{sec_further}

In this section we briefly outline a proof of Theorem~\ref{thextra}.
Since this is rather easy,
and not our main focus, we do not give full details.

First consider site percolation on the triangular lattice $T$. The definitions adapt in a very natural
way; for example, one can apply a linear map to map the vertex set of $T$ to $\Z^2$. Then nothing changes
in the definition of an essential enhancement except the underlying graph. To prove this case of
Conjecture~\ref{cmain}/Theorem~\ref{thextra} we need an analogue of Lemma~\ref{inside},
whose statement and proof we omit, and the following analogue of Conjecture~\ref{cpath}.
Here $\delta$ denotes the graph distance in $T$, and $B_r=B_r(0)=\{v:\delta(v,0)\le r\}$
where $0=(0,0)$ is the origin.

\begin{theorem}
For any $m>0$ there exists some $r>m$ such that the following holds.
Assume we have an induced path $P$ in the $2$-dimensional
triangular lattice through the origin $(0,0)$ joining points $a$ and $b$ that
lie outside $B_{r+2}$. Then by adding and deleting vertices
within $B_r$ we can obtain a set $S$ of vertices such that
$S\cap B_m=\{(i,0):i=-r,\dots,r\}$ with the property that any path joining $a$ to $b$ in $S$
goes via $(0,0)$.
\end{theorem}
\begin{proof}
Starting from $a$ consider the first (red) point $R_1$ at which $P$ enters $B_r$.
Similarly, let $G_1$ be the first (green) point of $P$ in $B_r$ that we reach starting from $b$.
Since $R_1$ and $G_1$ come before and after $0$ on the induced path $P$, they are not adjacent,
i.e., they are at graph distance at least $2$ in the hexagon~$S_r=B_r\setminus B_{r-1}$.
Remove all points of $P$ in $B_r$ except for $R_1$ and $G_1$.
It is easy to check that we can join $R_1$ to a point $u$ and $G_1$
to a point $v$ with $u,v\in S_{r-1}$ and $u$, $v$ at distance
at least $2$, so that $v$ is not a neighbour of $R_1$ and $u$ is not a neighbour
of $G_1$. Now $S_{r-1}$ is a cycle, so within $S_{r-1}$ we can proceed from $u$ moving
away from $v$ to reach the point $v'$ opposite $v$, say. Now it is easy to connect
$v'$ and $v$ by paths of length $2$ to two opposite points in $S_{r-3}$ and
then these points within $S_{r-3}$ to $(r-3,0)$ and $(-(r-3),0)$.
The result follows taking $r=m+4$, say.
\end{proof}

Finally, let us briefly discuss bond percolation on $\Z^d$. This turns out to be much easier
than site percolation for the following reason. In the arguments in the previous section we are
often faced with a task of the following form: connect point $a$ to point $b$ and point $c$ to point $d$
without connecting point $a$ to point $c$, and with some constraints as to how we can modify
an existing configuration. In either site or bond percolation the connections from $a$ to $b$
and from $c$ to $d$ will be open paths $P_1$ and $P_2$. In site percolation, such paths can relatively easily
`accidentally' connect $a$ to $c$ -- if any site on $P_1$ is adjacent to any site on $P_2$.
In bond percolation, $P_1\cup P_2$ contains an open path from $a$ to $c$ if and only if $P_1$
and $P_2$ share a vertex, and this is much easier to avoid.

For bond percolation there are several variants of the definition
of an enhancement (for example, the rule may
apply at every bond, or at every site). In all cases it is easy to prove the analogue
of Lemma~\ref{inside}, so to prove Theorem~\ref{thextra} we need the analogue of Conjecture~\ref{cpath}.
But this is essentially trivial! Indeed, given the two (red and green in our previous terminology)
paths $P_R$ and $P_G$ starting outside $\Binf_{r+2}$ and ending at (or near) the origin,
we may truncate them as follows: $P_R$ first hits $\Binf_r$ at some site $u$, along the edge $u'u$,
$u'\notin \Binf_r$. Similarly $P_G$ hits $\Binf_r$ at some site $v$ along an edge $v'v$.
Now delete all bonds inside $\Binf_r$, and also all bonds incident with $\Binf_r$
apart from $u'u$ and $v'v$. It remains only to connect $u$ and $v$ `cleanly' within $\Binf_r$: this will
never create any unwanted connections to the outside. But $\Sinf_r$ is $2$-connected,
so we may join $\{u,v\}$ and $\{e_r,-e_r\}$ by two vertex-disjoint paths within $\Sinf_r$.
Then join $e_r$ and $-e_r$ by a straight line.

\medskip
As mentioned earlier, given how simple the situation is for bond percolation, much the most
interesting open cases of Conjecture~\ref{cmain} are for site percolation on $\Z^d$, $d\ge 4$.
At the moment this seems to be a surprisingly difficult combinatorial problem.

\end{document}